\documentclass[11pt]{amsart}

\allowdisplaybreaks[4]
\newtheorem{proposition}{Proposition}[section]
\newtheorem{lemma}[proposition]{Lemma}

\newtheorem{theorem}[proposition]{Theorem}

\theoremstyle{definition}

\newtheorem{remark}[proposition]{Remark}

\newcommand{\thlabel}[1]{\label{th:#1}}
\newcommand{\thref}[1]{Theorem~\ref{th:#1}}
\newcommand{\selabel}[1]{\label{se:#1}}
\newcommand{\seref}[1]{Section~\ref{se:#1}}
\newcommand{\lelabel}[1]{\label{le:#1}}
\newcommand{\leref}[1]{Lemma~\ref{le:#1}}

\newcommand{\eqlabel}[1]{\label{eq:#1}}
\newcommand{\equref}[1]{(\ref{eq:#1})}

\def\ot{\otimes}

\newcommand{\Cc}{\mathcal{C}}

\newcommand{\Mm}{\mathcal{M}}

\newcommand{\Zz}{\mathcal{Z}}

\def\*C{{}^*\hspace*{-1pt}{\Cc}}
\def\text#1{{\rm {\rm #1}}}

\def\ul{\underline}

\def\equal#1{\smash{\mathop{=}\limits^{#1}}}

\usepackage{xypic}

\usepackage{color,amssymb,graphicx,amscd}

\input xy
\xyoption {all} \CompileMatrices

\begin{document}
\title[The uniqueness of braidings]
{The uniqueness of braidings on the monoidal category of
non-commutative descent data}

\author{A. L. Agore}
\address{Faculty of Engineering, Vrije Universiteit Brussel, Pleinlaan 2, B-1050 Brussels, Belgium}
\email{ana.agore@vub.ac.be and ana.agore@gmail.com}

\author{S. Caenepeel}
\address{Faculty of Engineering, Vrije Universiteit Brussel, Pleinlaan 2, B-1050 Brussels, Belgium}
\email{scaenepe@vub.ac.be}

\author{G. Militaru}
\address{Faculty of Mathematics and Computer Science, University of Bucharest, Str.
Academiei 14, RO-010014 Bucharest 1, Romania}
\email{gigel.militaru@fmi.unibuc.ro and gigel.militaru@gmail.com}


\thanks{A.L. Agore is research fellow ``Aspirant'' of FWO-Vlaanderen. S. Caenepeel is supported by
FWO project G.0117.10 ``Equivariant Brauer groups and Galois
deformations''. A.L. Agore and G. Militaru are supported by the
CNCS - UEFISCDI grant no. 88/05.10.2011 ''Hopf algebras and
related topics''.}

\begin{abstract}
Let $A$ be an algebra over a commutative ring $k$. It is known
that the categories of non-commutative descent data, of comodules
over the Sweedler canonical coring, of right $A$-modules with a
flat connection are isomorphic as braided monoidal categories to
the center of the category of $A$-bimodules. We prove that the
braiding on these categories is unique if there exists a
$k$-linear unitary map $E : A \to Z(A)$. This condition is
satisfied if $k$ is a field or $A$ is a commutative or a separable
algebra.
\end{abstract}

\maketitle

\section*{Introduction}
Braided monoidal categories play a central role in the
representation theory of quantum groups, Kac-Moody algebras,
quantum field theory, topological invariants to links, knots and
3-manifolds, or non-commutative differential geometry.

A natural problem is to classify all possible braidings on a given
monoidal category $\Cc$. The problem is far from being trivial as
we have to compute the class of all possible natural isomorphisms
$c_{C, D}: C\ot D\to D\ot C$, for all $C$, $D\in \Cc$, and this
depends heavily on the structure of the objects in $\Cc$. The
basic example is the following: braidings on the category of
representations of a bialgebra $H$ are parameterized by R-matrices
$R\in H\ot H$. A special role in the classification of all
braidings on a given monoidal category will be played by monoidal
categories $\Cc$ on which we have a \textit{unique} braided
structure. There are two typical examples of such monoidal
categories: the category of all sets $({\mathcal Set}, \times ,
\{*\} )$ and $(\Mm_k, -\ot_k - , k)$, the category of $k$-modules
over a commutative ring. The only braiding on this two categories
is the usual flip map. In \cite{ACM1} we examined braidings on the
category of bimodules over an algebra $A$. In most cases there is
no braiding at all; in particular situations, for example when $A$
is a central simple algebra over a field $k$, there is a unique
braiding, see \cite[Theorem 2.1, Cor. 2.7]{ACM1}.

In this note, we study braidings on the monoidal category
$\Mm^{A\ot A}$ of comodules over the Sweedler canonical $A$-coring
$A\ot A$. This category has several alternative descriptions: it
is isomorphic to the category of descent data $\ul{\rm
Desc}(A/k)$, to the category $\ul{\rm Conn}(A/k)$ of right
$A$-modules with a flat connection as defined in noncommutative
geometry \cite{tb06} and to the center $\Zz({}_A\Mm_A)$ of the
monoidal category of $A$-bimodules, all these isomorphisms can be
found in \cite[Theorem 2.10]{ACM2}.

The center $\Zz({}_A\Mm_A)$ is braided by construction, hence it
follows that $\Mm^{A\ot A}$, $\ul{\rm Conn}(A/k)$ and $\ul{\rm
Desc}(A/k)$ are also braided: the explicit description of this
braiding, called the canonical braiding, is given in
\cite[Corollary 2.11]{ACM2}, see also \cite[Lemma 2.2]{tb11}. The
aim of this note is to show that this canonical braiding is
unique; we will show this in \thref{unicitbr}, under the
assumption that there exists a $k$-linear unitary map $E : A \to
Z(A)$. This holds true if $k$ is a field or $A$ is a commutative
or a separable algebra over $k$.

\section{Preliminaries}\selabel{0.1}
Recall from \cite[Def. XI.2.1]{K} that a monoidal category is a
sixtuple $(\Cc, \ot,I, a,l,r)$, where $\Cc$ is a category, $\ot:\
\Cc\times \Cc\to \Cc$ is a bifunctor, $I$ is an object in $\Cc$,
and $a:\ \ot \circ (\ot \times id)\to \ot \circ (id \times \ot)$,
$l:\ \ot \circ (u \times id)\to id$, $r:\ \ot \circ (id \times
u)\to id$ are natural isomorphisms, such that certain coherence
conditions are satisfied. $\Cc$ is strict if $a$, $l$ and $r$ are
the identity natural transformations; McLane's coherence theorem
allows us to restrict attention to strict monoidal categories. A
braiding on $\Cc$ is a natural isomorphism $c:\ \ot\to \ot\circ
\tau$ satisfying the following compatibilities:
\begin{enumerate}
\item[(B1)] \qquad $c_{U, \, V\ot W} = (Id_{V} \ot c_{U, \,
W})\circ (c_{U, \, V}\ot Id_{W})$

\item[(B2)] \qquad  $c_{U\ot V, \, W} = (c_{U, \, W}\ot
Id_{V})\circ (Id_{U} \ot c_{V, \, W})$
\end{enumerate}
for all $U,V,W\in \Cc$, where $\tau: \mathcal{C} \times
\mathcal{C} \to \mathcal{C} \times \mathcal{C}$ is the flip functor.
A braiding $c$ is called a symmetry if $c_{U, V}^{-1} =
c_{V, U}$, for all $U$, $V\in \Cc$. A (symmetric) braided category
$(\Cc, c)$ is a monoidal category $\Cc$ equipped with a
(symmetric) braiding $c$. More details on braided categories can
be found in \cite{JS}, \cite{K}.\\

Let $A$ be a $k$-algebra over a commutative ring $k$ and let
$Z(A)$ be the center of $A$. Unadorned $\ot$ means $\ot_k$ and
$A^{(n)}$ will be a shorter notation for the $n$-fold tensor
product $A\ot\cdots\ot A$. ${}_A\Mm_A = ({}_A\Mm_A, - \ot_A - ,
A)$ is the $k$-linear monoidal category of $A$-bimodules. An
$A$-coring $C$ is a coalgebra in ${}_A\Mm_A$. A right $C$-comodule
is a right $A$-module $M$ together with a right $A$-linear map
$\rho:\ M\to M\ot_A C$ satisfying the coassociativity and the
counit axioms. $\Mm^C$ is the category of right $C$-comodules and
right $C$-colinear maps. For further details on corings and
comodules, we refer to \cite{BrzezinskiWisbauer}. An important
example of an $A$-coring is Sweedler's canonical coring $C = A\ot
A$. Identifying $(A\ot A)\ot_A (A\ot A)\cong A^{(3)}$, we will
view the comultiplication as a map $\Delta:\ A^{(2)}\to A^{(3)}$
given by the formula $\Delta(a\ot b) = a\ot 1\ot b$; the counit
$\varepsilon: A^{(2)} \to A$ is given by $\varepsilon (a\ot b) =
ab$. For a right $A$-module $M$, we can identify $M\ot_A (A\ot A)
\cong M\ot A$. A right $A\ot A$-comodule is then a right
$A$-module $M$ together with a $k$-linear map $\rho:\ M\to M\ot
A$, denoted by $\rho(m) = m_{[0]}\ot m_{[1]}$ (summation is
implicitly understood), satisfying the compatibility conditions
\begin{eqnarray}
m_{[0]}m_{[1]}&=&m;\eqlabel{3.1.1}\\
\rho(m_{[0]})\ot m_{[1]}&=& m_{[0]}\ot 1\ot m_{[1]};\eqlabel{3.1.2}\\
\rho(ma)&=&m_{[0]}\ot m_{[1]}a\eqlabel{3.1.3}
\end{eqnarray}
for all $m\in M$ and $a\in A$. A morphism in $\Mm^{A\ot A}$ is a
right $A$-module map $f: M \to N$ such that for any $m\in M$
\begin{equation}\eqlabel{morphSW}
f(m)_{[0]} \ot f(m)_{[1]} = f(m_{[0]}) \ot m_{[1]}
\end{equation}
The category of right $A\ot A$-comodules is denoted by $\Mm^{A\ot
A}$. There is an adjunction pair $(F := - \ot A, \, G := (-)^{{\rm
co}(A\ot A)} )$ between $\Mm_k$ and $\Mm^{A\ot A}$ defined as
follows:
$$
F = - \ot A : \Mm_k \to \Mm^{A\ot A}, \quad G = (-)^{{\rm co}(A\ot
A)} : \Mm^{A\ot A} \to \Mm_k
$$
where for any $k$-module $V$, $F(V) = V \ot A$ is a right $A\ot
A$-comodule with the right $A$-module structure given by the right
multiplication on $A$ and the coaction
\begin{equation}\eqlabel{cofree}
\rho_{V\ot A} : V\ot A \to V \ot A \ot A, \quad \rho_{V\ot A} (v
\ot a) := v \ot 1_A \ot a
\end{equation}
for all $v\in V$ and $a\in A$. If $(M, \rho) \in \Mm^{A\ot A}$,
then $ G (M) = M(-)^{{\rm co}(A\ot A)} := \{ m\in M \, | \, \rho
(m) = m \ot 1_A \}$. $A$ will be viewed as a right $A\ot
A$-comodule via the regular right action given by the
multiplication on $A$ and the right coaction is given by
\begin{equation}\eqlabel{coactA}
\rho_{A} : A\to A \ot A, \quad \rho_{A} (a) :=  1_A \ot a
\end{equation}
for all $a\in A$. Cipolla's noncommutative descent data
\cite{Cipolla} are precisely right $A\ot A$-comodules and
Cipolla's version of the Faithfully Flat Descent Theorem can be
reformulated as follows: $(F, G)$ is a pair of inverse
equivalences if $A$ is faithfully flat over $k$ \cite[Proposition
109]{book}.

\section{Braidings on the category of $A\ot A$-comodules}\selabel{3}
A right $A\ot A$-comodule $M$ is also a $k$-module, so we can
consider $F(M) = M\ot A\in \Mm^{A\ot A}$ via
\begin{equation}\eqlabel{coactVA}
(m\ot a) \, b := m\ot ab \qquad {\rm} \quad \rho : M\ot A \to M\ot
A\ot A, \quad \rho (m \ot a) := m\ot  1_A \ot a
\end{equation}
for all $m\in M$, $a$, $b\in A$. It easily follows from
\equref{3.1.2} and \equref{3.1.3} that the coaction $\rho:\ M \to
M\ot A$ is a morphism in $\Mm^{A\ot A}$. In \cite[Prop.
2.2]{ACM2}, we observed that a right $A\ot A$-comodule $M$ carries
a {\sl left} $A$-module structure given by
\begin{equation}\eqlabel{leftA}
a\cdot m = m_{[0]} \, a \, m_{[1]}
\end{equation}
for all $a\in A$ and $m\in M$. In particular, $M\ot A$ is a left
$A$-module, with left $A$-action $a \cdot (m\ot b) = (m\ot 1_A) a
b = m \ot ab$. Then $\rho:\ M\to M\ot A$ is left $A$-linear;
indeed, for any $a \in A$ and $m\in M$ we have that
$$
\rho (a \cdot m) = \rho (m_{[0]} a m_{[1]}) = m_{[0][0]} \ot
m_{[0][1]} a m_{[1]}\, \equal{\equref{3.1.2}}\,  m_{[0]} \ot a
m_{[1]} = a\cdot \rho(m)
$$
Take $M\in \Mm^{A\ot A}$. It follows from \equref{3.1.1} that the
right $A$-action $\mu_{M}: M \ot A \to M$ is a right $A$-linear
splitting map of the coaction $\rho: \ M \to M\ot A$. However,
$\mu_{M}$ is in general not left $A$-linear since
$$
a \cdot \mu_M (m \ot b) = a \cdot (m b) = m_{[0]} a m_{[1]} b
$$
while
$$
\mu_M ( a\cdot (m \ot b)) = \mu_M (m \ot ab) = m (ab).
$$
This is the major drawback in our attempt to prove the uniqueness
of the braiding on $\Mm^{A\ot A}$; in the proof of \thref{unicitbr}, we will need
a left $A$-linear splitting map for $\rho$. We give sufficient conditions for its existence
in the next lemma.

\begin{lemma}\lelabel{2.1}
Let $A$ be a $k$-algebra, and assume that there exists a
$k$-linear map $E : A \to Z(A)$ such that $E(1_A) = 1_A$. For any
$M\in \Mm^{A\ot A}$, the map $\mu^E_M : M \ot A \to M$, $\mu^E_M
(m\ot a) = m_{[0]} E (m_{[1]}) a$ is a left $A$-linear splitting
map for the coaction $\rho : M \to M \ot A$.
\end{lemma}

\begin{proof}
We first show that $\mu^E_M$ is left $A$-linear. For all $a,b\in A$ and $m\in M$, we have
\begin{eqnarray*}
&&\hspace*{-2cm}
b \cdot \mu^E_M (m\ot a) = b \cdot m_{[0]} E ( m_{[1]}) a =
 m_{[0][0]} b \, m_{[0][1]} E ( m_{[1]}) a \\
&\equal{\equref{3.1.2}}& m_{[0]} b \, E ( m_{[1]}) \, a
= m_{[0]} E ( m_{[1]}) \, b \, a
= \mu^E_M (m\ot ba) = \mu^E_M (b\cdot (m\ot a)).
\end{eqnarray*}
Finally we show that $\mu^E_M\circ \rho = {\rm Id}_M$:
$$
\mu^E_M  (m_{[0]} \ot m_{[1]}) = m_{[0][0]} E( m_{[0][1]}) m_{[1]}
= m_{[0]} E(1_A) m_{[1]} = m,
$$
for all $m\in M$.
\end{proof}

\begin{theorem}\thlabel{braidedcomod}(\cite[Cor. 2.11]{ACM2})
For a $k$-algebra $A$, the category $(\Mm^{A\ot A}, \, - \ot_A -,
\, A)$ of right comodules over Sweedler's canonical coring is
symmetric monoidal. For $M$, $N\in \Mm^{A\ot A}$, the coaction
$\rho$ on $M\ot_A N$ is
\begin{equation}\eqlabel{coact}
\rho: M\ot_A N \to M\ot_A N \ot A, \quad \rho(m\ot_A n) =
m_{[0]}\ot_A n_{[0]}\ot m_{[1]}n_{[1]}
\end{equation}
for all $m\in M$, $n\in N$. The unit is $A$ and the symmetry $c$
is given by the maps
\begin{equation}\eqlabel{canbraid}
c_{M, \, N} : M\ot_A N \to N\ot_A M, \quad c_{M, \, N} (m\ot_A n)
= n_{[0]} \ot_A m \, n_{[1]}
\end{equation}
for any $M$, $N \in \Mm^{A\ot A}$, $m\in M$, $n\in N$.
\end{theorem}

\begin{proof}
This follows from the fact that $\Mm^{A\ot A}$ is isomorphic to the center
$\Zz({}_A\Mm_A)$ of the category of $A$-bimodules, which is braided monoidal,
we refer to \cite{ACM2} for full detail. Let us show that the braiding is a symmetry:
we have that
\begin{eqnarray*}
c_{N, \, M} \circ c_{M, \, N} (m \ot_A n) &=& c_{N, \, M} (n_{[0]}
\ot_A m \, n_{[1]}) \stackrel{\equref{3.1.3}} = m_{[0]} \ot_A
n_{[0]} m_{[1]} n_{[1]}\\ &\stackrel{\equref{leftA}}=& m_{[0]}
\ot_{A} m_{[1]} \cdot n = m_{[0]} m_{[1]} \ot_{A}
n \\
&\stackrel{\equref{3.1.1}} =& m\ot_A n
\end{eqnarray*}
for all $m\in M$ and $n\in N$.
\end{proof}

\begin{proposition}
Let $A$ be an algebra over a commutative ring $k$. We know that
$(\Mm^{A\ot A}, \ot_A, A)$ is a monoidal category, with a
canonical symmetry \equref{canbraid}. The functor $F=-\ot A:\
\Mm_k\to \Mm^{A\ot A}$ is a symmetric monoidal functor.
\end{proposition}

\begin{proof}
For $M$, $N\in \Mm_k$, we have natural isomorphisms
$$ \varphi_{0}: A \to F(k) = k \ot A, \qquad
\varphi_{0}(a) = 1 \ot a$$
$$\varphi_{M, N}: F(M) \ot_{A} F(N) = (M \ot A)\ot_{A}(N \ot A) \to F(M \ot N) = M \ot N \ot
A$$ $$\varphi_{M, N}(m \ot a \ot_{A} n \ot b) = m \ot n \ot ab$$
Straightforward computations show that $F$ together with this
family of natural isomorphisms is a monoidal functor. In order to
show that $F$ preserves the symmetry, we have to show that the
diagram
$$\xymatrix{(M\ot A)\ot_A(N\ot A)\ar[d]_{\varphi_{M, N}}\ar[rr]^{c_{M \ot A, N \ot A}}&&
(N\ot A)\ot_A(M\ot A)\ar[d]^{\varphi_{N, M}}\\
M\ot N\ot A\ar[rr]_{\tau_{M, N} \ot {A}}&& N \ot M \ot A}$$
commutes, for all $M,N\in \Mm_k$. $\tau$ is the symmetry on
$\Mm_k$, and is given by the switch map. Using \equref{canbraid},
we compute
\begin{eqnarray*}
&&\hspace*{-2cm} (\varphi_{N, M} \circ c_{M \ot A, N \ot
A})\bigl((m \ot a) \ot_{A}
(n \ot b)\bigl)\\
&=& \varphi_{N, M}((n \ot b)_{[0]} \ot_{A} (m \ot
a) (n \ot b)_{[1]}) \\
&=& \varphi_{N, M}((n \ot 1_{A}) \ot_{A} (m \ot ab))=
 n \ot m \ot ab;\\
 &&\hspace*{-2cm}
 \bigl((\tau_{M,N} \ot {A})\circ \varphi_{M,N}\bigl)\bigl((m
\ot a) \ot_{A} (n \ot b)\bigl)\\
&=&(\tau_{M,N} \ot Id_{A})(m \ot n \ot ab) = n \ot m \ot ab,
\end{eqnarray*}
for all $a,b \in A$, $m \in M$ and $n \in N$, as needed.
\end{proof}

\begin{remark}
We note that the forgetful functor $F: \Mm^{A\ot A} \to {}_A\Mm_A$
is a strict monoidal functor. In case ${}_A\Mm_A$ is a braided
monoidal category (see \cite[Theorem 2.1]{ACM1}) the functor $F$
is not however a braided monoidal functor. To see this, we
consider $K$ to be a commutative ring such that $2$ is invertible
in $K$, and $A={}^aK^b$ the generalized quaternion algebra having
$\{1,i,j,k\}$ as a $K$-basis, where $a$, $b$ are invertible
elements in $K$. Then we can write down the explicit formula for
the (unique) braiding on ${}_A\Mm_A$ by using the $R$-matrix
described in \cite[Example 2.10]{ACM1}. It follows that $F$ is not
a braided functor by considering the appropriate diagram for the
pair of objects $M = N : =A$ in $\Mm^{A\ot A}$ and checking that
it is not commutative in $i \ot_{A} j$.
\end{remark}

In order to prove the uniqueness of the braiding on $\Mm^{A\ot A}$, we will need the
following Lemma.

\begin{lemma}\lelabel{abrad}
Let $A$ be a $k$-algebra and $a\in A$. Then the natural
transformation $c$ given by
$$
c^a_{M, \, N} : M\ot_A N \to N\ot_A M, \quad c^a_{M, \, N} (m\ot_A
n) = n_{[0]} \ot_A m \, n_{[1]} \, a
$$
is a braiding on the monoidal category $(\Mm^{A\ot A}, \, - \ot_A
-, \, A)$ if and only of $a = 1_A$.
\end{lemma}

\begin{proof}
If $c^a$ is a braiding then $c^a_{A, A} : A \ot_A A \to A\ot_A A$,
$c^a_{A, A} (x\ot_A y) = 1_A \ot_A xy a$ is an isomorphism of
right $A\ot A$-comodules and, a fortiori, of right $A$-modules,
and this implies that $a$ has a left inverse in $A$. On the other
hand, evaluating both sides of (B2) to $1_A \ot 1_A \ot_{A} 1_A
\ot 1_A \ot_{A} 1_A \ot 1_A$ in the situation where $U = V = W =
A\ot A$, we find that $1_A \ot 1_A \ot 1_A \ot a = 1_A \ot 1_A \ot
a \ot a$. Multiplying the tensor factors, we obtain that $a^2 = a$
and hence $a = 1_A$ since $a$ has a left inverse.
\end{proof}

Now we can state and prove the main result of this note.

\begin{theorem}\thlabel{unicitbr}
Let $A$ be a $k$-algebra, and assume that there exists a $k$-linear map
$E : A \to Z(A)$ such that $E(1_A) = 1_A$. Then there is precisely one
braiding on the monoidal category $(\Mm^{A\ot A}, \, - \ot_A -, \, A)$,
namely the canonical braiding defined in \equref{canbraid}.
\end{theorem}

\begin{proof}
Let $c$ be a braiding on $\Mm^{A\ot A}$. For morphisms $f: M\to
M'$ and $g: N \to N'$ in $\Mm^{A\ot A}$, the following diagram commutes,
by the naturality of $c$:
$$
\begin {CD}
M \ot_A N  @> c_{M, N} >> N\ot_A M\\
@V f\ot_A g VV @ VV g\ot_A f V\\
M'\ot_A N' @> c_{M', N'} >> N'\ot_A M'
\end{CD}
$$
As we have seen at the end of \seref{0.1}, $A\ot A=F(A)\in \Mm^{A\ot A}$,
and $A\ot A$ is also a left $A$-module, via \equref{leftA}, which takes the form
$a \cdot (b \ot c) = b \ot ac$.\\
The identification $A^{(3)} \cong A^{(2)}\ot_A A^{(2)}$ transports the isomorphism
$c_{A^{(2)}, \, A^{(2)}}: A^{(2)}\ot_A A^{(2)}\to  A^{(2)}\ot_A
A^{(2)}$
to an isomorphism $ \gamma : A^{(3)} \to A^{(3)}$. Then $c_{A^{(2)}, \, A^{(2)}}$ can be
computed from $\gamma$ as follows:
$$
c_{A^{(2)}, \, A^{(2)}} (a \ot b \ot_A a' \ot b') = c_{A^{(2)}, \,
A^{(2)}} (a \ot 1_A \ot_A b\cdot (a' \ot b') ) = \gamma (a \ot a'\ot b b'),
$$
where we identified $A^{(2)}\ot_A A^{(2)}$ and $A^{(3)}$ in the last identity.
$\gamma$ is completely determined by the map
$$
\delta : A^{(2)} \to  A^{(3)}, \quad \delta (a \ot b) = \gamma (a
\ot b \ot 1_A).
$$
Since $\gamma$ is right $A$-linear, we have
$$
\gamma (a \ot b \ot c) = \gamma (a \ot b \ot 1_A) c = \delta (a\ot b) c.
$$
Now we adopt the temporary notation:
$$\delta (a\ot b) = \sum \delta^1 (a\ot b) \ot \delta^2 (a\ot b) \ot
\delta^3 (a\ot b) \in A^{(3)}.$$
Then we have that
$$
\gamma (a \ot b \ot c) =  \sum \delta^1 (a\ot b) \ot \delta^2
(a\ot b) \ot \delta^3 (a\ot b) \, c
$$
and
\begin{equation}\eqlabel{20}
c_{A^{(2)}, \, A^{(2)}} (a \ot 1_A \ot_A a' \ot b') = \sum
\delta^1 (a\ot a') \ot 1_A \ot_A \delta^2 (a\ot a') \ot \delta^3
(a\ot a') \, b',
\end{equation}
for all $a,a',b\in A$. $c_{A^{(2)}, \, A^{(2)}}$ is a right
$A\ot A$-colinear map, this means that the following diagram commutes:
$$
\begin {CD}
A^{(2)} \ot_A A^{(2)} @> c_{A^{(2)}, A^{(2)}} >> A^{(2)}\ot_A A^{(2)}\\
@V \rho VV @ VV \rho V\\
A^{(2)}\ot_A A^{(2)} \ot A @> c_{A^{(2)}, A^{(2)}}\ot A>> A^{(2)}\ot_A
A^{(2)} \ot A
\end{CD}
$$
where $\rho:\ A^{(2)} \ot_A A^{(2)} \to A^{(2)}\ot_A A^{(2)} \ot
A$, the right $A\ot A$-coaction defined in \equref{coact}, is
given by the formula
$$\rho (a \ot 1_A \ot_A a' \ot b') = a\ot 1_A \ot_A a' \ot 1_A \ot
b'.$$
Evaluating the diagram at $a \ot 1_A \ot_A a' \ot 1_A$, we find that
$\delta$ satisfies the equation
\begin{equation}\eqlabel{21}
\sum \delta^1 (a\ot a') \ot \delta^2 (a\ot a') \ot \delta^3 (a\ot
a') \ot 1_A = \delta^1 (a\ot a') \ot \delta^2 (a\ot a') \ot 1_A
\ot \delta^3 (a\ot a'),
\end{equation}
for all $a,a'\in A$. In fact, it can be shown easily that the right $A\ot A$-colinearity
of $c_{A^{(2)}, \, A^{(2)}}$ is equivalent to \equref{21}, but this will not be needed.\\
For $a\in A$, the map $f_a : A^{(2)} \to A^{(2)}$
given by $f_a (x \ot y) := ax \ot y$ is a morphism in $\Mm^{A\ot
A}$. The naturality of $c$ implies that the following diagram commutes, for
all $a,b\in A$:
$$
\begin {CD}
A^{(2)} \ot_A A^{(2)} @> c_{A^{(2)}, A^{(2)}} >> A^{(2)}\ot_A A^{(2)}\\
@V f_a \ot_A f_b VV @ VV f_b\ot_A f_a V\\
A^{(2)}\ot_A A^{(2)} @> c_{A^{(2)}, A^{(2)}} >> A^{(2)}\ot_A
A^{(2)}
\end{CD}
$$
Evaluating this diagram at $1_A \ot 1_A \ot_A 1_A \ot c$, we obtain that
\begin{equation}\eqlabel{25}
c_{A^{(2)}, \, A^{(2)}} (a \ot 1_A \ot_A b \ot c) = \sum bs^1 \ot
1_A \ot_A a \, s^2 \ot s^3 \, c
\end{equation}
for any $a$, $b$, $c\in A$, where $\delta (1_A \ot 1_A) = \sum s^1
\ot s^2 \ot s^3 \in A^{(3)}$. This implies that $\delta$ is completely determined
by $\delta (1_A \ot 1_A)$:
\begin{equation}\eqlabel{26}
\delta (a \ot b) = \sum bs^1 \ot  a \, s^2 \ot s^3.
\end{equation}
Combining  \equref{21} and  \equref{26},
$$
\sum a' \, s^1 \ot a \, s^2 \ot s^3 \ot 1_A = \sum a' \, s^1 \ot a
\, s^2 \ot 1_A \ot s^3
$$
for all $a,a'\in A$. In particular, for $a = a' = 1_A$ we
obtain
$$
\sum s^1 \ot  s^2 \ot s^3 \ot 1_A = \sum  s^1 \ot s^2 \ot 1_A \ot
s^3
$$
Multiplying the second and the third tensor factor, we find that
 $s = \sum s^1 \ot s^2 s^3 \ot 1_A$. We conclude that
there exists an element $R = \sum R^1 \ot R^2 \in A\ot A$ such
that $s = R\ot 1_A$. Then we have that
\begin{equation}\eqlabel{27}
\delta (a \ot b) = \sum b \, R^1 \ot  a \, R^2 \ot 1_A
\end{equation}
\begin{equation}\eqlabel{28}
c_{A^{(2)}, \, A^{(2)}} (a \ot 1_A \ot_A b \ot c) = \sum b \, R^1
\ot 1_A \ot_A a \, R^2 \ot c
\end{equation}
for all $a,b,c\in A$. We can easily prove that $c_{A^{(2)},
\, A^{(2)}}$ as defined in \equref{28} is an isomorphism if and only if
$R$ is invertible in the algebra $A\ot A$, but this will not be needed.\\
For $M \in \Mm_A$ and $m\in M$, the map $f_m : A^{(2)}
\to M\ot A$, $f_m (a \ot b) = ma \ot b$, is a morphism in
$\Mm^{A\ot A}$, where $M\ot A$ is viewed as a right $A\ot
A$-comodule via \equref{coactVA}. From the naturality of $c$, it follows
that the following diagram commutes, for all $M,N \in
\Mm_A$, $m\in M$ and $n\in N$:
$$
\begin {CD}
A^{(2)} \ot_A A^{(2)} @> c_{A^{(2)}, A^{(2)}} >> A^{(2)}\ot_A A^{(2)}\\
@V f_m \ot_A f_n VV @ VV f_n\ot_A f_m V\\
M\ot A \ot_A N\ot A @> c_{M\ot A, N\ot A} >> M\ot A \ot_A N\ot A
\end{CD}
$$
Evaluating this diagram at $1_A \ot 1_A \ot_A 1_A \ot a$ and
using \equref{28} we obtain that
\begin{equation}\eqlabel{29}
c_{M\ot A, \, N\ot A} (m \ot 1_A \ot_A n \ot a) = \sum n \, R^1
\ot 1_A \ot_A m \, R^2 \ot a
\end{equation}
for all $m\in M$, $n\in N$ and $a\in A$. This means that the braiding $c$ is
completely determined in all cofree objects $M\ot A$ of the
category $\Mm^{A\ot A}$ by the element $R \in A\ot A$.\\
For $M\in \Mm^{A\ot A}$, the coaction $\rho_M : M \to M \ot A$ is a morphism in $\Mm^{A\ot A}$,
so the following diagram commutes, again by the naturality of $c$:
$$
\begin {CD}
M \ot_A N @> c_{M, N} >> N\ot_A M \\
@V \rho_M \ot_A \rho_N VV @ VV \rho_N\ot_A \rho_M V\\
M\ot A \ot_A N\ot A @> c_{M\ot A, N\ot A} >> M\ot A \ot_A N\ot A
\end{CD}
$$
Evaluating this diagram at $m\ot_A n$, we find that
\begin{eqnarray}
&&\hspace*{-2cm}
(\rho_N\ot_A \rho_M) ( c_{M, N} (m\ot_A n)) = c_{M\ot A, N\ot
A} (m_{[0]} \ot m_{[1]} \ot_A  n_{[0]} \ot n_{[1]}  ) \nonumber\\
&=& c_{M\ot A, N\ot A} (m_{[0]} \ot 1_A \ot_A  m_{[1]} \cdot (n_{[0]} \ot n_{[1]} ) )\nonumber \\
&=& c_{M\ot A, N\ot A} (m_{[0]} \ot 1_A \ot_A n_{[0]} \ot m_{[1]}  n_{[1]}  )\nonumber \\
&\equal{\equref{29}}& \sum n_{[0]} R^1 \ot 1_A \ot_A m_{[0]}
R^2 \ot m_{[1]}  n_{[1]}.\eqlabel{30}
\end{eqnarray}
The multiplication map $\mu_N$ is right $A$-linear and splits $\rho_N$, and the
map $\mu_M^E$ from \leref{2.1} is left $A$-linear and splits $\rho_M$. This implies
that $\mu_N \ot_A \mu^E_M$ splits $\rho_N\ot_A\rho_M$. Applying $\mu_N \ot_A \mu^E_M$
to \equref{30}, we obtain that
\begin{eqnarray*}
c_{M, N} (m\ot_A n) &=& \sum n_{[0]} R^1 \ot_A \mu^E_M ( m_{[0]}
R^2 \ot m_{[1]} n_{[1]} )\\
&=& \sum n_{[0]} R^1 \ot_A m_{[0][0]} E( m_{[0][1]} R^2 ) m_{[1]}
n_{[1]} \\
&\equal{\equref{3.1.2}}& \sum n_{[0]} R^1 \ot_A m_{[0]} E(R^2 ) m_{[1]}
n_{[1]} \\
&=& \sum n_{[0]} R^1 \ot_A m_{[0]}  m_{[1]} E(R^2 )
n_{[1]} \\
&\equal{\equref{3.1.1}}& \sum n_{[0]} R^1 \ot_A m E(R^2 ) n_{[1]} \\
&=& \sum n_{[0]} \ot_A R^1 \cdot (m E(R^2 ) n_{[1]}) \\
&=& \sum n_{[0]} \ot_A m_{[0]} R^1 m_{[1]} E(R^2 )
n_{[1]},
\end{eqnarray*}
for any $m\in M$ and $n\in N$. In the special case where  $M = N =
A\ot A$, we evaluate this formula to $a \ot 1_A \ot_A b \ot c$.
Using \equref{28}, we obtain that
$$
\sum b R^1 \ot 1_A \ot_A a R^2 \ot c = \sum b \ot 1_A \ot_A a\ot R^1 E(R^2) c
$$
for all $a,b,c \in A$. In particular, for $a = b = c = 1_A$,
we find
$$
\sum R^1 \ot R^2 \ot 1_A = \sum 1_A \ot 1_A \ot R^1 E(R^2).
$$
Multiplying the second and the third tensor factors, we obtain that $ R =
\sum 1_A \ot R^1 E(R^2)$, so we can conclude that  $R = 1_A \ot \alpha$, for some
$\alpha \in A$. Therefore, we obtain:
$$
c_{M, N} (m\ot_A n) = n_{[0]} \ot_A m_{[0]} m_{[1]} E(\alpha )
n_{[1]} = n_{[0]} \ot_A m E(\alpha ) n_{[1]} = n_{[0]} \ot_A m
n_{[1]} E(\alpha ).
$$
It then follows from \leref{abrad} that $c$ is the canonical
symmetry given by \equref{canbraid}.
\end{proof}

Let us finally examine the existence of a unitary $k$-linear map
$A\to Z(A)$.

\begin{proposition}
Let $A$ be an algebra over a commutative ring $k$. There exists
a unitary $k$-linear map  $E:\ A\to Z(A)$ in each of the following situations:
\begin{enumerate}
\item $A$ is commutative;
\item $k$ is a field;
\item $A$ is an augmented algebra, for example a bialgebra;
\item $A$ is a separable $k$-algebra.
\end{enumerate}
\end{proposition}

\begin{proof}
The first three cases are obvious. Let $A$ be a separable algebra
with separability idempotent $e = e^1\ot e^2$ (summation
understood), i.e.
\begin{equation}\eqlabel{sepalgebra}
a e^1 \ot e^2 =  e^1 \ot e^2 a \qquad {\rm and} \qquad e^1 e^2 = 1
\end{equation}
for all $a\in A$. The map $E: A\to A$, $E(a) = e^1ae^2$ meets the
requirements: $E(a)\in Z(A)$ follows from the centrality condition
and $E(1_A)=1_A$ follows from the normality condition in
\equref{sepalgebra}.
\end{proof}

We end our paper with the following question: does there exist a commutative ring $k$
and a $k$-algebra $A$ for which there exists a second braiding on $\Mm^{A\ot A}$?

\end{document}